\RequirePackage{fix-cm}

\documentclass[smallcondensed]{svjour3}     % onecolumn (ditto)
\smartqed  % flush right qed marks, e.g. at end of proof
\usepackage{graphicx}
\usepackage{amssymb,amsmath}
\usepackage{tikz}
  \usepackage[all]{xy}
  \usepackage{bm}
  \usetikzlibrary{matrix}
\DeclareMathOperator{\Uprec}{{\Uparrow}}
\DeclareMathOperator{\Dprec}{{\Downarrow}}
\DeclareMathOperator{\RI}{\mathcal{RI}}
\DeclareMathOperator{\RF}{\mathcal{RF}}
\DeclareMathOperator{\pt}{pt}
\DeclareMathOperator{\R}{\mathfrak{F}}

\DeclareMathOperator{\Prim}{Prim}
\DeclareMathOperator{\cl}{cl}

\DeclareMathOperator{\uOf}{Clop^{\uparrow}}
\DeclareMathOperator{\dOf}{Clop^{\downarrow}}
\DeclareMathOperator{\nin}{\not\in}

\DeclareMathOperator{\Int}{int}
\renewcommand{\int}{\Int}

\DeclareMathOperator{\F}{\mathfrak{F}}
\DeclareMathOperator{\End}{End}

\spnewtheorem{corollary}[theorem]{Corollary}{\bfseries}{\itshape}
\spnewtheorem{lemma}[theorem]{Lemma}{\bfseries}{\itshape}
\spnewtheorem{proposition}[theorem]{Proposition}{\bfseries}{\itshape}
\spnewtheorem{definition}[theorem]{Definition}{\bfseries}{}
\spnewtheorem{remark}[theorem]{Remark}{\bfseries}{}
\counterwithin{theorem}{section}

\begin{document}

\title{Towards a Gleason cover for compact pospaces%\thanks{Grants or other notes
}
\subtitle{}


\author{Laurent De Rudder        \and
        Georges Hansoul %etc.
}


\institute{Laurent De Rudder \at
              D\'epartement de math\'ematiques (B37)\\
Universit\'e de Li\`ege\\Belgium \\
\email{l.derudder@uliege.be} \\
 %  \\
}

\date{Received: date / Accepted: date}

\maketitle

\begin{abstract}
We establish a new category equivalent to compact pospaces, and which extend the equivalence between compact Hausdorff spaces and Gleason spaces. As a corollary of this equivalence, we obtain in particular, that every compact pospace is the quotient of an f-space.
\end{abstract}
\section*{Introduction}

The category \textbf{DeV} of de Vries' compingent algebras \cite{deVries} can be considered among neighbouring categories in a network of Stone-like and Gelfand-like dualities involving the categories \textbf{KHaus} of compact Hausdorff spaces, \textbf{GlSp} of Gleason spaces, \textbf{KrFrm} of compact regular frames and \textbf{\emph{C}$^\star$-alg} of $C^\star$-algebras. 
\begin{center}
\begin{tikzpicture}
\node(KHaus) at (-1,0) {\textbf{KHaus}};
\node(DeV) at (1,0) {\textbf{DeV}};
\node(C) at (0,3.1) { \textbf{\emph{C}$^\star$-alg}};
\node(GlSp) at (1.6,1.9) { \textbf{GlSp}};
\node(KrFrm) at (-1.6,1.9) {\textbf{KrFrm}};
\draw[-, >=latex] (KHaus) to (DeV) ;
\draw[-, >=latex] (DeV) to(GlSp);
\draw[-, >=latex] (KrFrm) to(KHaus);
\draw[-, >=latex] (KrFrm) to (C);
\draw[-, >=latex] (GlSp) to (C);
\end{tikzpicture}
\end{center}

In particular, the category \textbf{KHaus} and \textbf{GlSp} are equivalent, as  established in \cite{Sourabh}. This was  first observed via the composition of the dualities between \textbf{KHaus} and \textbf{DeV} and between \textbf{DeV} and \textbf{GlSp}, then with a direct description. 

Categories of this base network were later generalized in different papers. Indeed, Bezhanishvili and Harding extended in \cite{Guramtriang} the dualities and equivalences between \textbf{KHaus}, \textbf{KrFrm} and \textbf{DeV} to dualities and equivalence between the categories \textbf{StKSp} of stably compact spaces, \textbf{StKFrm} of stably compact frames and \textbf{PrFrm} of proximity frames. As for the duality between \textbf{KHaus} and \textbf{\emph{C}$^\star$-alg}, a real version of the duality, given in \cite{Gurambal}, was extended in \cite{DeHaGelfand} to a duality between \textbf{KPSp} of compact pospaces and the category \textbf{usbal} of Stone semirings. We refer to \cite{Guramtriang} and \cite{DeHaGelfand} for the relevant definitions.
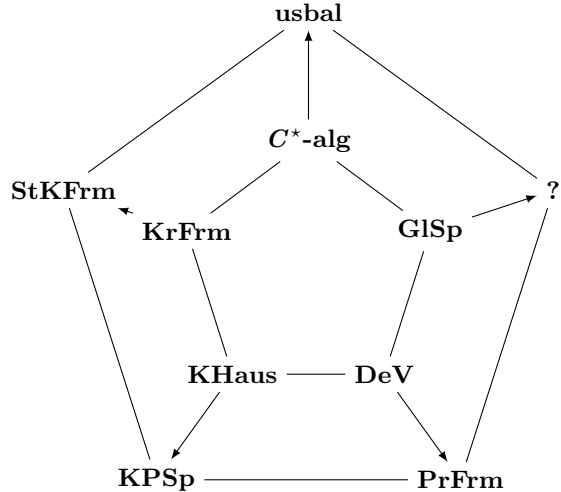
\begin{center}
\begin{tikzpicture}
\node(KHaus) at (-1,0) {\textbf{KHaus}};
\node(KPSp) at (-2,-1.4) {\textbf{KPSp}};
\node(DeV) at (1,0) {\textbf{DeV}};
\node(PrFrm) at (2,-1.4) {\textbf{PrFrm}};
\node(C) at (0,3.1) { \textbf{\emph{C}$^\star$-alg}};
\node(GlSp) at (1.6,1.9) { \textbf{GlSp}};
\node(KrFrm) at (-1.6,1.9) {\textbf{KrFrm}};
\node(StKFrm) at (-3.22,2.44) {\textbf{StKFrm}};
\node(OGlSp) at (3.22,2.44) {\textbf{?}};
\node(usbal) at (0,4.8) {\textbf{usbal}};
\draw[-, >=latex] (KHaus) to (DeV) ;
\draw[-, >=latex] (DeV) to(GlSp);
\draw[-, >=latex] (KrFrm) to(KHaus);
\draw[-, >=latex] (KrFrm) to (C);
\draw[-, >=latex] (GlSp) to (C);
\draw[-, >=latex] (KPSp) to (PrFrm) ;
\draw[-, >=latex] (PrFrm) to(OGlSp);
\draw[-, >=latex] (StKFrm) to(KPSp);
\draw[-, >=latex] (StKFrm) to (usbal);
\draw[-, >=latex] (OGlSp) to (usbal);
\draw[->, >=latex] (C) to (usbal);
\draw[->, >=latex] (GlSp) to (OGlSp);
\draw[->, >=latex] (KrFrm) to (StKFrm);
\draw[->, >=latex] (KHaus) to (KPSp);
\draw[->, >=latex] (DeV) to (PrFrm);
\end{tikzpicture}
\end{center}

 The aim of this paper is to complete the extensions initiated in \cite{Guramtriang} and \cite{DeHaGelfand} to the category \textbf{GlSp}. We point out that this extension process follows the same spirit as passing from Boolean algebras to distributive lattices, and from Stone spaces to Priestley spaces in the zero-dimensional setting (from the Boolean to the distributive setting as we shall often say in this paper).

The methodology goes as follows. First, we will establish on Priestley spaces the counterpart of proximity relations on lattices. The road was well paved by Castro and Celani in \cite{Castro}, where the dual of a quasi-modal lattice (a generalized proximity frame, but with a different class of morphisms) was already established as Priestley spaces endowed with an increasing closed binary relation. The obtained topological structures will be named ordered Gleason spaces and will be the objects of a category whose morphisms are binary specific relations and not usual maps (as it is already the case in the Boolean setting \cite{Sourabh}). Then, since duals of proximity frames in \cite{Guramtriang} were stably compact spaces, we will spend a few words on how to describe them as compact pospaces. Finally, following Bezhanishvili steps in \cite{StonebyDV}, we will show how to obtain directly the compact po-space dual to a Proximity frame via the latter's Priestley dual.  

\section{Preliminaries} \label{Section_Prelim}

In this section, we recall previous dualities which are essential for this paper, mainly for the sake of establishing  notations that will be used throughout the rest of the paper.

\subsubsection*{Priestley duality}
We begin with the celebrated Priestley duality \cite{Priestley1} and its characterization to frames in \cite{Pultr} through a suitable separation property. 

First of all, if $(X,\leq,\tau)$ is an ordered topological space, we denote by $\tau^\uparrow$ (resp. $\tau^\downarrow$) the topology of open upsets (resp. open downsets) of $\tau$. 

In particular, if $(X,\leq,\tau)$ is a Priestley space, it is well known that $\tau^\uparrow$ (resp. $\tau^\downarrow$) is generated by the clopen upsets (resp. clopen downsets) of $X$, which we denote by $\uOf(X)$ (resp. $\dOf(X)$). Moreover, $\uOf(X)$ (or simply $L$, should the context cause no confusion) is a distributive lattice when ordered by inclusion. Finally, if $f : X \longrightarrow Y$ is an increasing continuous function between Priestley space, then 
 \[ \uOf(f) : \uOf(Y) \longrightarrow \uOf(X) : O \longmapsto f^{-1}(O) \] is a 
lattice morphism.

%Moreover, $\tau$ is generated by $\uOf(X) \cup \dOf(X)$. Finally, the sets $\uOf(X)$ and $\dOf(X)$ ordered by inclusion are bounded distributive lattices. 

On the other hand, if $L$ is a bounded distributive lattice, we denote by $\Prim(L)$ (or more simply $X$) its set of prime filters, ordered by inclusion  and endowed by the topology generated by 
 \[ \lbrace \eta(a) \mid a \in L \rbrace \cup \lbrace \eta(a)^c \mid a \in L \rbrace, \] where 
 \[ \eta(a) := \lbrace x \in \Prim(L) \mid  x \ni a \rbrace. \] Then $\Prim(L)$ is a Priestley space and $\eta$ is a lattice isomorphism between $L$ and $\uOf(\Prim(L))$. Moreover, if $h : L \longrightarrow M$ is a lattice morphism then 
\[ \Prim(h) : \Prim(M) \longrightarrow \Prim(L) : x \longmapsto h^{-1}(x) \] is an increasing continuous function. The functors $\Prim$ and $\uOf$ establish a duality between the categories \textbf{DLat}, of bounded distributive lattices, and \textbf{Priest}, of Priestley spaces.

To continue, let us recall that a \emph{frame} is a complete lattice $L$  which satisfies the \emph{join infinite distributive law}: for every subset $S \subseteq L$ and every $a\in L$, we have
 \[ a \wedge \bigvee S = \bigvee \lbrace a \wedge s \mid s \in S \rbrace. \] Furthermore, a lattice morphism $h : L \longrightarrow M$ between two frames is a \emph{frame morphism} if it preserves arbitrary joins. 
 
  \begin{lemma}[\cite{Pultr}]\label{lem_to_f_space} Let $L$ be a frame and $(X,\leq,\tau)$ be its Priestley dual.
 \begin{enumerate}
 \item If $O \in \tau^\uparrow$, then its closure in $\tau$, denoted by $\cl(O)$, is an open upset. 
 \item If $S$ is a subset of $\uOf(X)$, then $ \bigvee S = \cl\left( \bigcup \lbrace O \mid O \in S \rbrace \right)$.
 \item The map $\eta : a \longmapsto \eta(a)$ is a frame morphism.
 \end{enumerate}
 \end{lemma}
 
 Refering to this result, an \emph{f-space} is a Priestley space $(X,\leq,\tau)$ which satisfies the first item of Lemma \ref{lem_to_f_space} and an increasing continuous function $f : X \longrightarrow Y$ is an \emph{$f$-function} if $f^{-1}\left(\cl(O)\right) = \cl\left(f^{-1}(O)\right)$ for all $O \in \tau^\uparrow$.
 
Pultr and Sichler proved in \cite{Pultr} that Priestley duality reduces to a duality between the categories \textbf{Frm} of frames and \textbf{FSp} of $f$-spaces. 
 
 \subsubsection*{Proximity frames}
 
 The second duality we recall was established by Bezhanishvili and Harding in \cite{Guramtriang}, it can be seen as a generalization to frames and stably compact spaces (see \cite[Definition VI-6.7.]{Compendium}) of de Vries duality. 
 
 \begin{definition}\label{def_proximity_frame_and_morphism} A \emph{proximity frame} is a pair $(L,\prec)$ where $L$ is a frame and $\prec$ is a \emph{proximity relation}, i.e. a binary relation on $L$ such that\footnote{The seemingly peculiar way used to denote the properties of $\prec$ (and the absence of S7) stems from the works on subordination and (pre-)contact algebras, see for instance \cite{Sourabh}, \cite{DeHa} or \cite{Koppelberg}.} 
 \begin{itemize}
 \item $\prec$ is a \emph{subordination relation}
 \begin{enumerate}
  \item[S1.] $0 \prec 0$ and $1 \prec 1$,
 \item[S2.] $a \prec b,c$ implies $a \prec b \wedge c$,
 \item[S3.] $a,b \prec c$ implies $a \vee b \prec c$,
 \item[S4.] $a \leq b \prec c \leq d$ implies $a \prec d$,
 \end{enumerate}
\item which has the following additional properties
\begin{enumerate}
  \item [S5.] $a = \bigvee \lbrace b \in L \mid b \prec a \rbrace$,
 \item[S6.] $ a \prec b$ implies $a \leq b$,
 \item[S8.] $a \prec b$ implies that $a \prec c \prec b$ for some $c$.
\end{enumerate}
 \end{itemize}
 For the sake of convenience, we often identify the pair $(L,\prec)$ with its underlying frame $L$. 
 
 If $S$ is a subset of $L$, we define $\Uprec S := \lbrace b \in L\mid \exists s \in S : s \prec b \rbrace $ ( $\Dprec S$ is defined dually). As usual, for an element $a \in L$, we write $\Uprec a$ instead of $\Uprec \lbrace a \rbrace$.
 \end{definition}
 \begin{definition}
A \emph{proximity morphism} is a map $h:L \longrightarrow M$ between two proximity frames such that:
\begin{enumerate}
\item[H0.] $h$ is a \emph{strong meet-hemimorphism}:
\begin{enumerate}
\item $h(1) = 1$ and $h(0) = 0$,
\item $h(a \wedge b) = h(a) \wedge h(b)$;
\end{enumerate}
\item[H1.] $a_1 \prec b_1$ and $a_2 \prec b_2$ implies $h(a_1 \vee a_2) \prec h(b_1) \vee h(b_2)$;
\item[H2.] $h(a) = \bigvee \lbrace h(b) \mid b \prec a \rbrace$.
\end{enumerate}
If $h : L \longrightarrow M$ and $g: M \longrightarrow N$ are proximity morphisms, their \emph{composition} is defined by
\[ g \star h : L \longrightarrow N : a \longmapsto \bigvee \lbrace g(h(b)) \mid b \prec a \rbrace. \]  
 We denote by \textbf{PrFrm} the category of proximity frames endowed with proximity  morphisms. 
 \end{definition}
 
 \begin{definition}\label{def_ends} If $L$ is a proximity frame, a \emph{round filter} of $L$ is a lattice filter $\F$ such that $\F = \Uprec \F$. We denote by $\mathcal{RF}(L)$ the set of all round filters of $L$.
 
 An \emph{end} is a round filter $p$ such that for every round filters $\F_1,\F_2$, we have $\F_1 \cap \F_2 \subseteq p$ if and only if $\F_1 \subseteq p$ or $\F_2 \subseteq p$. We denote by $\End(L)$ (or only by $P$) the set of all ends of $L$. 
 \end{definition}
 
 The ends of Definition \ref{def_ends} will now play a role similar to the one of prime filters in Priestley duality. Indeed, endowed with the topology generated by the sets of the form 
 \begin{equation}\label{eq_def_of_mu} \mu(a) := \lbrace p \in \End(L) \mid p \ni a \rbrace, \end{equation} $\End(L)$ is a stably compact space. Moreover, if $h: L \longrightarrow M$ is a proximity frame, then 
 \[ \End(h) : \End(M) \longrightarrow \End(L) : p \longmapsto \Uprec h^{-1}(p) \] is a proper continuous function. 
 
 On the topological side, if $(P,\tau)$ is a stably compact space, then $\tau:= \Omega(P)$ ordered by inclusion is a proximity frame when endowed with the relation $\prec$ defined by $O \prec V$ if and only if $O \subseteq K \subseteq V$ for some compact subset $K$. Furthermore, if $f : P \longrightarrow Q$ is a proper continuous function between two stably compact spaces, then 
 \[ \Omega(f) : \Omega(Q) \longrightarrow \Omega(P) : O \longmapsto f^{-1}(O) \] is  a proximity morphism. 
 
 Now, the functors $\End$ and $\Omega$ establish a duality between \textbf{PrFrm} and the category \textbf{StKSp} of stably compact spaces (see \cite[Theorem 4.18]{Guramtriang}).  
 
 \section{Priestley duality for proximity frames}
 
 In addition to the duality between \textbf{PrFrm} and \textbf{StKSp}, we can provide a modal-like duality between \textbf{PrFrm} and a category of f-spaces endowed with a particular binary relation $R$. Following the taxonomy of \cite{StonebyDV}, we name \emph{ordered Gleason spaces} the pairs $f$-spaces/relations obtained. At the objects level, we can rely on the works previously done in \cite{Castro} for quasi-modal lattices\footnote{The relations between proximity/subordination relations and quasi-modal operator is well discussed for instance in \cite{CelaniResume}.} and in \cite{Sourabh} for the Boolean setting. Hence, most of the proofs are left to the reader. 
 
 \begin{definition}\label{def_ordered_Gleason_space} An \emph{ordered Gleason space} is a triple $(X,\leq,R)$ where $(X,\leq)$ is an $f$-space and $R$ is a binary relation on $X$ satisfying the following properties:
 \begin{enumerate}
\item $R$ is closed in $X^2$;
\item $x \leq y \mathrel{R} z \leq t$ implies $x \mathrel{R} t$;
\item $R$ is a pre-order;
\item For every $O \in \uOf(X)$, we have $O = \cl\left( R[-,O^c]^c \right)$.
\end{enumerate}
An equivalent definition is given by substituting 2 with 
\begin{enumerate}
\item[2'.] $ x \leq y$ implies $x \mathrel{R} y$.
\end{enumerate}
 \end{definition}
 
 \begin{remark}\label{rem_useful_starting_rem}
 Let us highlight some observations and introduce notations that we freely use in the rest of the paper.
 \begin{itemize}
 \item   Let $R$ be a binary relation on an arbitrary set $X$:
 \begin{enumerate}
 \item If $E$ is a subset of $X$, we note 
 \[ R[-,E] := \lbrace x \mid \exists y \in E : x \mathrel{R} y \rbrace \text{ and } R[E,-] := \lbrace x \mid \exists y \in E : y \mathrel{R} x \rbrace. \]
 For an element $x \in X$, we note $R[-,x]$ instead of $R[-,\lbrace x \rbrace]$. Note that, if $(L,\prec)$ is a proximity frame, then we have $ \Uprec x = {\prec}[x,-]$.
 \item If $E$ and $F$ are subsets of $X$, then 
 \[R[-,E] \subseteq F \text{ if and only if } R[F^c,-] \subseteq E^c. \]
 \item If $X$ is a topological space, $R$ is closed in $X^2$ and $F$ is a closed subset of $X$, then $R[-,F]$ and $R[F,-]$ are closed.
 \end{enumerate}
 \item  Let $L$ be a distributive lattice and $S$ an arbitrary subset of $L$. We define 
\[ F_S := \lbrace x \in \Prim(L) \mid S \subseteq x \rbrace. \]
Remark that $ F_S = \bigcap \lbrace \eta(a) \mid a \in S \rbrace$ so that $F_S$ is a closed (and hence a compact) subset of $\Prim(L)$.
 \end{itemize}
 \end{remark}

The future duality between proximity frames and ordered Gleason spaces is now obtained as follows. Let $(L,\prec)$ be a proximity frame, its dual is given by $(X,R)$ where $X=\Prim(L)$ is the Priestley dual of $L$ and $R$ is the binary relation on $X$ defined by 
\begin{equation}\label{eq_def_of_R}
x \mathrel{R} y \text{ if and only if } \Uprec x \subseteq y.
\end{equation}
Let us highlight the fact that equivalent definitions of the relation $R$ are given by 
\begin{equation}\label{rem_alternative_def_of_R} \Uprec x \subseteq \Uprec y \text{ or } \Dprec y^c \subseteq \Dprec x^c \text{ or } \Dprec y^c \subseteq x^c. \end{equation}

\begin{lemma} Endowed with the relation $R$ defined in \eqref{eq_def_of_R}, $\Prim(L)$ is an ordered Gleason space. Furthermore, for every $a,b \in L$, we have 
\[ a \prec b \text{ if and only if } R[\eta(a),-] \subseteq \eta(b). \]
\end{lemma}
\begin{proof}[Sketch of the proof]
To prove Items 1 and 2 of Definition \ref{def_ordered_Gleason_space}, one just has to use the subordination part of a proximity relation (see Definition \ref{def_proximity_frame_and_morphism}). Also, one can show that $R$ is reflexive if and only if $\prec$ satisfies S6 and transitive if and only $\prec$ satisfies S8. Let us prove item 5 (which is equivalent to S5.). We have $a = \bigvee \lbrace b \mid b \prec a \rbrace$ if and only if $\eta(a) = \eta\left( \bigvee \lbrace b \mid b \prec a \rbrace \right)$. Then, by \cite[Theorem 1.5]{Pultr}, it comes that
 \begin{align*}
 \eta\left( \bigvee \lbrace b \mid b \prec a \rbrace \right) &=  \cl \left( \bigcup \lbrace \eta(b) \mid b \prec a \rbrace \right ) \\
 & = \cl \left( \bigcup \lbrace \eta(b) \mid R[\eta(b),-] \subseteq \eta(a) \rbrace \right) \\
 &= \cl \left( \bigcup \lbrace \eta(b) \mid \eta(b) \subseteq R[-,\eta(a)^c]^c \rbrace \right).
 \end{align*}
 Finally, since $R[-,\eta(a)^c]^c$ is an open upset , it follows that 
 \[ \eta\left( \bigvee \lbrace b \mid b \prec a \rbrace \right) = \cl \left( R[-,\eta(a)^c]^c \right), \] and the conclusion is clear.
\end{proof}

On the other hand, let $(X,\leq,R)$ be an ordered Gleason space, its dual is given by $(L,\prec)$ where $L:=\uOf(X)$ is the Priestley dual of $X$ and $\prec$ is the binary relation on $L$ defined by 
\begin{equation}\label{eq_def_of_prec} O \prec U \text{ if and only if } R[O,-] \subseteq U. \end{equation}

\begin{lemma} Endowed with the relation $\prec$ defined in \eqref{eq_def_of_prec}, $\uOf(X)$ is a proximity frame. 
\end{lemma}

To conclude the section, it remains to determine the counterpart of the proximity morphisms on Gleason spaces. Let $L$ and $M$ be proximity frames and $X$ and $Y$ their respective Priestley duals. If $h : L \longrightarrow M$ is a meet-hemimorphism, then the relation $\rho \subseteq Y \times X$ defined by 
\begin{equation}\label{eq_def_of_rho} y \mathrel{R} x \text{ if and only if } h^{-1}(y) \subseteq x \end{equation} satisfy the following conditions:
\begin{enumerate}
\item $y_1 \leq y_2 \mathrel{\rho} x_1 \leq x_2$ implies $y_1 \mathrel{\rho} x_2$,
\item $\rho$ is closed in $Y \times X$,
\item  $O \in \uOf(X)$ implies $\rho[-,O^c]^c \in \uOf(Y)$.
\end{enumerate}

Since in our case, we have \emph{strong} meet-hemimorphism, the relation $\rho$ also satisfies 
\begin{enumerate}
\item[4.] for every $y \in Y$, there exists $x \in \rho[y,-]$.  
\end{enumerate}
We call such a relation $\rho$ a \emph{strong meet-hemirelation}. By \cite[Lemma 2]{Sofronie}, we know that strong meet-hemimorphism are in correspondence with strong meet-hemirelation. Hence, it remains to characterize the properties H1 and H2 of proximity morphisms. A key concept  towards this characterization is defined below.

\begin{definition} Let $(X,\leq,R)$ be an ordered Gleason space and $S$ a subset of $X$. An element $x \in S$ is said to be \emph{$R$-minimal in $S$} if for every $y \in S$, $y \mathrel{R} x$ implies $x \mathrel{R} y$. 
\end{definition}

\begin{proposition}\label{prop_exist_R_min} Let $(X,R)$ be an ordered Gleason space and $F$ be a closed subset of $X$, then for every element $x \in F$ there exists an element $y$  $R$-minimal in $F$ such that $y \mathrel{R} x$.
\end{proposition}
\begin{proof} We follow the lines of the proof for po-sets (see for instance \cite[Proposition VI.5-3.]{Compendium})
Let us define a chain of $(X,R)$ to be a subset $C$ of $X$ such that for every $x,y\in C$, we have $x \mathrel{R} y$ or $y \mathrel{R} x$. 

We denote by $\mathfrak{C}$ the set of chains $C$ satisfying $x \in C \subseteq F$, ordered by inclusion. We have that $\mathfrak{C}$ is non-empty (by reflexivity of $R$, it contains the chain $\lbrace x \rbrace$) and a classical argument suffices to prove it is also inductive. Hence, $\mathfrak{C}$ admits a maximal element $M$. 

Since $\lbrace R[-,z] \cap F \mid z \in M \rbrace$ is a family of closed sets which satisfies the finite intersection properties (because $M$ is a chain contained in $F$ and $R$ is a pre-order),  we know by compactness that there exists an element $y \in F$ such that $y \mathrel{R} z$ for all $z \in M$. 

Now, suppose that $t$ is an element of $F$ such that $t \mathrel{R} y$. By transitivity, we have that $\lbrace t \rbrace \cup M$ is a chain of $\mathfrak{C}$.  By maximality of $M$, we have $t \in M$ and, therefore, we have $y \mathrel{R} t$, so that $y$ is indeed $R$-minimal in $F$, as required.
\end{proof}

 Let us highlight that the notion of $R$-minimal element is also present in the Boolean setting, while hidden. Indeed, in the Boolean case, the relation $R$ turns out to be an equivalence relation, so that every element is actually $R$-minimal. 

\begin{proposition}\label{Prop_towards_ofc} Let $h: L \longrightarrow M$ be a strong meet-hemimorphism between two proximity frame and $\rho \subseteq Y \times X$ its associated strong hemi-relation: 
\begin{enumerate}
\item $h$ satisfies H1 if and only if for every $y_1, y_2 \in Y$, every $x_1$ R-minimal in $\rho[y_1,-]$ and every $x_2 \in X$, we have 
%\[ \left. \begin{array}{r}
%y_1 \mathrel{R} y_2 \\
%y_1 \mathrel{\rho} x_1 \\
%y_2 \mathrel{\rho} x_2
%\end{array}  \right\rbrace \Rightarrow x_1 \mathrel{R} x_2, \]
\[ x_1 \mathrel{\rho^{-1}} y_1 \mathrel{R} y_2 \mathrel{\rho} x_2 \text{ implies } x_1 \mathrel{R} x_2. \]
\item $h$ satisfies H2 if and only if $ \rho[-,O^c] = \int\left(\rho[-,R[-,O^c]]\right)$ for every $O \in \uOf(X)$.
\end{enumerate}
\end{proposition}

The proof of Item 2 is almost identical to the one in the Boolean case. Therefore, we redirect the reader to \cite[Lemma 6.11]{Sourabh} for more details. The proof of Item 1 requires additional results. In the meantime, we name \emph{ordered forth condition} (shortened as \emph{ofc}) and \emph{de Vries condition} (shortened as \emph{dvc}) the conditions of the first and the second item of Proposition \ref{Prop_towards_ofc}.

Before we start, let us note that 
\begin{equation}\label{eq_rem} \rho[-,\eta(a)^c]^c = \eta(h(a)). \end{equation}
Indeed, it is clear that $\eta(h(a)) \subseteq \rho[-,\eta(a)^c]^c.$ Now, suppose that $y \in \rho[-,\eta(a)^c]^c$. Then, for every $x \in \eta(a)^c$, we have that $h^{-1}(y) \not\subseteq x$. Hence, for all $x \in \eta(a)^c$, we have that $h(a_x) \in y$ and $a_x \nin x$ for some $a_x \in L$. In particular, $\lbrace \eta(a_x)^c \mid x \in \eta(a)\rbrace$ is an open cover of $\eta(a)^c$ which is compact. Then, we now that there exist $x_1,\ldots,x_n \in \eta(a)^c$ such that 
\[ \eta(a_{x_1} \wedge \cdots \wedge a_{x_n})^c \supseteq \eta(a)^c. \] Moreover, we have ($y$ is a filter) 
\[ y \ni h(a_{x_1}) \wedge \cdots \wedge h(a_{x_n}) = h(a_{x_1} \wedge \cdots \wedge a_{x_n}) \geq h(a) \]
and the conclusion is clear.

\begin{proposition}\label{prop_ok_funct}
Let $L,M$ be two proximity frames, $h: L \longrightarrow M$ be a proximity morphism , $y$ a prime filter of $M$ and $x$ a prime filter which is $R$-minimal in $\rho[y,-]$. Then, we have 
\[ \Uprec\left( h^{-1}\left( \Uprec y \right) \right) = \Uprec x. \]
\end{proposition}
\begin{proof}
On the one hand,  $ h^{-1}\left( \Uprec y \right)  \subseteq x$ follows from $x \in \rho[y,-]$. Consequently, we have $\Uprec\left( h^{-1}\left( \Uprec y \right) \right) \subseteq \Uprec x$.

On the other hand, suppose that $\Uprec x \not\subseteq \Uprec\left( h^{-1}\left( \Uprec y \right) \right)$. Then, there exist $a_1 \in x$ and $b \in L$ such that $a \prec b$ and $b \nin \Uprec\left( h^{-1}\left( \Uprec y \right) \right)$. By the properties of proximity relations, we know that $a_1  \prec c_1 \prec d_1 \prec e_1 \prec b$ for some $c_1,d_1,e_1 \in L$. In particular, we have $h(d_1) \prec h(e_1)$ and $e_1 \nin h^{-1}\left( \Uprec y \right)$. Therefore, we also have that $h(d_1) \nin y$. 

In order to obtain an absurdity and conclude the proof, we are going to invalidate the $R$-minimality of $x$ in $\rho[y,-]$.

We first prove that 
\begin{equation}\label{eq_mid_proof} h^{-1}(y) \cap \langle c_1 \cup \Dprec x^c \rangle_{id} = \emptyset, \end{equation} where $\langle c_1 \cup \Dprec x^c \rangle_{id}$ is the lattice ideal generated by $c_1 \cup \Dprec x^c$. Suppose this is not the case. Then, there exist $a_2,c_2 \in L$ and $d_2 \in x^c$ such that $h(a_2) \in y$, $c_2 \prec d_2$ and $a_2 \leq c_1 \vee c_2$. It follows from the properties of $h$ that 
\[y \ni h(a_2) \leq h(c_1 \vee c_2) \prec h(d_1) \vee  h(d_2). \]
Now, since $y$ is a prime filter and $h(d_1) \nin y$, we have that $h(d_2) \in y$. Hence, we have $d_2 \in h^{-1}(y) \subseteq x$, which is absurd. Consequently, \eqref{eq_mid_proof} is satisfied and we have $h^{-1}(y) \subseteq z$, $c_1 \nin z$ and $\Uprec z \subseteq x$ for some prime filter $z$. In other words, we have $z \mathrel{R} x$ and $y \mathrel{\rho} z$. Now, by $R$-minimality of $x$ in $\rho[y,-]$, it follows that $x \mathrel{R} z$. Hence, in particular, we should have 
\[ c_1 \in \Uprec a_1 \subseteq \Uprec x \subseteq z, \] which is absurd. 
\end{proof}

We now have the required result to finish the proof of Proposition \ref{Prop_towards_ofc}.

\begin{proof}[Proof of Proposition \ref{Prop_towards_ofc}]
For the only if part, suppose that $h^{-1}(y_1) \subseteq x_1$, $h^{-1}(y_2) \subseteq x_2$ and $\Uprec y_1 \subseteq y_2$. In particular, by Proposition \ref{prop_ok_funct}, we have $\Uprec \left( h^{-1}\left( \Uprec y_1 \right) \right) = \Uprec x_1$. It comes that
\[ \Uprec x_1 = \Uprec \left( h^{-1}\left( \Uprec y_1 \right) \right) \subseteq h^{-1}\left( \Uprec y_1 \right) \subseteq h^{-1}(y_2) \subseteq x_2,\] or, in other words, that $x_1 \mathrel{R} x_2$, as required.

For the if part, let $a_1,a_2,b_1$ and $b_2$ be elements of $L$ such that $a_1 \prec b_1$ and $a_2 \prec b_2$. To prove that $h$ satisfies H1 is to prove that 
\[ R[\eta(h(a_1 \vee a_2)),-] \subseteq \eta(h(b_1)) \cup \eta(h(b_2)). \]
We can use \eqref{eq_rem} to rewrite this inclusion as 
\[ \underbrace{R[\rho[-,\eta(a_1 \vee a_2)^c]^c,-]}_{:=A} \subseteq \underbrace{\rho[-,\eta(b_1)^c]^c \cup \rho[-,\eta(b_2)^c]^c}_{:=B}. \]
Let $y_2 \in A$. Then, there exists $y_1$ such that $y_1 \mathrel{R} y_2$ and such that $ \rho[y_1,-] \subseteq  \eta(a_1 \vee a_2)$. Moreover, by Proposition \ref{prop_exist_R_min}, we know that there exists a filter $x_1$ $R$-minimal in $\rho[y_1,-]$. Hence, we may suppose, without loss of generality, that $a_1 \in x_1$. Let $x_2$ be a prime filter such that $y_1 \mathrel{\rho} x_2$. By the ofc, we know that $x_1 \mathrel{R} x_2$ and it follows that \[ b_1 \in \Uprec a_1 \subseteq \Uprec x_1 \subseteq x_2 .\] Hence, we proved that for every $x_2$ such that $y_2 \mathrel{\rho} x_2$, we have $x_2 \in \eta(b_1)$, that is $y_2 \in \rho[-,\eta(b_1)^c]^c \subseteq B$, as required.
\end{proof}

Now that we characterized the strong meet-hemirelation that stemmed from proximity morphisms, we have to determine how to compose them to actually obtain category dual to \textbf{PrFrm}. As it was already noted in \cite{Sourabh}, the rule of composition of meet-hemirelations is not easily described, even in the Boolean setting and we must rely on their associated meet-hemimorphisms.

\begin{definition}\label{def_comp_of_relations} Let $\rho_1$ and $\rho_2$ be meet-hemirelations and $h_1$, $h_2$ their associated meet-hemimorphisms. We define the \emph{composition} $\rho_1 \star \rho_2$ as the meet-hemirelation associated to $h_2 \star h_1$. 
\end{definition}

With all of the above observations, the next definition and theorem come as no surprise.

\begin{definition} We denote by \textbf{OGlSp} the category whose objects are ordered Gleason spaces and whose morphisms are strong meet-hemirelations which satisfy the ofc and the dvc, with the composition of Definition \ref{def_comp_of_relations}. For the record, let us note that the identity morphisms in \textbf{OGlSp} are given by the order relations of the ordered Gleason spaces.
\end{definition}

\begin{theorem}\label{thm_dual_Gle_Pr}
The categories \textbf{OGlSp} and \textbf{PrFrm} are dual to each other.
\end{theorem}

Of course, as direct corollary of Theorem \ref{thm_dual_Gle_Pr} and \cite{Guramtriang}, the categories \textbf{OGlSp} and \textbf{StKSp} are equivalent. The scope of the next section is to describe directly this equivalence.. However, since this paper is "ordered-minded", we swap the category \textbf{StKSp} for its equivalent category \textbf{KPSp} of compact pospaces, also sometimes called Nachbin spaces.

\section{Compact pospaces}\label{Section3}

\begin{definition} A \emph{compact pospace} is a triple $(P,\pi,\leq)$ where $(P,\pi)$ is a compact space and $\leq$ is an order relation on $P$ which is closed in $P^2$.  We denote by \textbf{KPSp} the category of compact pospaces and continuous monotone maps.
\end{definition}

The equivalence between \textbf{KPSp} and \textbf{StKSp} is almost folklore (see for instance \cite[Section VI-6]{Compendium}). We recall here the basic facts. 

If $(P,\tau)$ is a stably compact space, then $(P,\pi,\leq_\tau)$ is a compact pospace where $\pi$ is the patch topology 
associated to $\tau$ and $\leq_\tau$ is the canonical order on $(P,\tau)$, that is $p \leq_\tau q$ if and only if $p \in \cl_\tau(\lbrace q \rbrace)$. In addition, we have $\pi^\uparrow = \tau$ and $\pi^\downarrow$ is the co-compact topology associated to $\tau$, that is the compact saturated sets of $\tau$. On the other hand, if $(P,\pi,\leq)$ is a compact pospace, then $(P,\pi^\uparrow)$ is a stably compact space. With these consideration in mind, we can describe the ends space $\End(L)$ of a proximity frame $L$ as a compact pospace.

\begin{proposition}\label{prop_carac_of_P_as_KPSP} Let $L$ be a proximity frame and $P:=\End(L)$ its ends space.
\begin{enumerate}
\item For $p,q \in P$, $p \leq q$ if and only if $p \subseteq q$,
\item The topology $\pi^\uparrow$ is generated by the sets $\mu(a)$ for $a \in L$ (see \eqref{eq_def_of_mu}),
\item The closed elements of $\pi^\downarrow$ are given by the sets of the form 
\[ K_{\F}:= \lbrace p \in P \mid p \subseteq \F \rbrace; \] for some round filter $\F$.
\end{enumerate}
\end{proposition}
\begin{proof} Item 2 is immediate. For item 1, we have
\begin{align*}
& p \leq q \\
\iff & p \in \cl_\tau(\lbrace q \rbrace) = \cap \lbrace \mu(a)^c \mid q \in \mu(a)^c \rbrace \\
\iff & \forall a \in L : a \nin q \Rightarrow a \nin p \\
\iff &q^c \subseteq p^c \iff p \subseteq q. \end{align*}
To prove Item 3, we use several results established in \cite{Guramtriang}.
\begin{enumerate}
\item From \cite[Lemma 4.14]{Guramtriang}, there is a homeomorphism $\alpha$ from $\End(L)$ to $\pt\mathcal{RI}(L)$\footnote{The points of a frame $M$ are frame morphisms $g$ from $M$ into \textbf{2}, the two elements frame. We denote by $\mbox{pt}(M)$ the set of all points of $M$. A \emph{round ideal} of a proximity frame is a lattice ideal $\mathfrak{I}$ such that $\Dprec \mathfrak{I} = \mathfrak{I}$. The set $\mathcal{RI}(L)$ of all round ideals of $L$ is a frame when order by inclusion. } given by 
\[ \alpha : \pt(\RI(L)) \longrightarrow \End(L) : g \longmapsto p_g:= \lbrace a \in L \mid g(\Dprec a) = 1 \rbrace. \]
\item From \cite[Remark 4.21]{Guramtriang}, there is a bijection between $\RF(L)$ and the Scott-open filters of $\RI(L)$ given by 
\[ \R \longmapsto \lbrace \mathfrak{I} \in \RI(L) \mid \R \cap \mathfrak{I} \neq \emptyset \rbrace. \]
\item From \cite[Theorem II.1.20]{Compendium}, there is an order-reversing bijection between Scott-open filter of $\RI(L)$ and the the compact saturated sets of $\pt(\RI(L))$ which is given by 
\[ \mathbb{F} \longmapsto  \lbrace p \in \pt(\RI(L)) \mid \forall \mathfrak{I} \in \mathbb{F} :  g(\mathfrak{I}) = 1 \rbrace. \]
\end{enumerate}
Consequently, there is a bijection between the compact saturated sets of $\pt(\RI(L))$ and $\RF(L)$ which is given by 
\[ \R \longmapsto \lbrace g \in \pt(\RI(L)) \mid \forall \mathfrak{I} \in \RI(L) :   \mathfrak{I}\cap\R \neq \emptyset \Rightarrow g(\mathfrak{I}) = 1 \rbrace. \] Then, since $\alpha$ is a homeomorphism, we now that there is a bijection between the compact saturated sets of $\End(L)$ and $\RF(L)$ given by 
\[ \R \longmapsto \lbrace p \in \End(L) \mid  \underbrace{\forall I \in \RI(L):\mathfrak{I} \cap \R \neq \emptyset \Rightarrow \mathfrak{I} \cap p \neq \emptyset}_{=\star} \rbrace . \] 
Finally, to conclude the proof, let us show that the condition $\star$ is equivalent to $\R \subseteq p$. Clearly, if $\R \subseteq p$, then $\star$ is satisfied. Now, suppose that $\R \not\subseteq p$. Then, there exists an element $a \in \R \setminus p$. In particular, we have that $\Dprec a$ is a round ideal such that $\Dprec a \cap p = \emptyset$ and, since $\R$ is round, such that $\Dprec a \cap \R \neq \emptyset$ such that the condition $\star$ is not satisfied. 
\end{proof}

Now, we focus on how $\End(L)$ relates with $\Prim(L)$. A first immediate remark  is that for every round filter $\F$ and every prime filter $x$, we have 
\begin{equation}\label{eq_prime_and_round_filters}
\F \subseteq x \Leftrightarrow \F \subseteq \Uprec x.
\end{equation}
A second step is undertaken in the following lemma.

\begin{lemma}\label{lem_precx=round} Let $L$ be a proximity frame. For every prime filter $x \in \Prim(L)$, $\Uprec x$ is an end of $L$. 
\end{lemma}
\begin{proof}
It is clear that $\Uprec x$ is a filter but, by S8 of Definition \ref{def_proximity_frame_and_morphism}, it is also a round filter. Moreover, if $\F_1$ and $\F_2$ are round filters such that $\F_1 \cap \F_2 \subseteq \Uprec x$. In particular, this implies that $\F_1 \cap \F_2 \subseteq x$. Now, $x$ being a prime filter, we know that $\F_1 \subseteq x$ or $\F_2 \subseteq x$. It then follows from \eqref{eq_prime_and_round_filters} that $\Uprec x$ is an end.

%Now, since $x$ is prime\footnote{If $x$ is a prime filter, then it is meet-prime in the lattice of filters.}
\end{proof}

Our goal now is to prove that every end is of the form $\Uprec x$ for some prime filter $x$. We start with the next proposition.

\begin{proposition}\label{prop_round_filter_and_r_increasing_set} Let $L$ be a proximity frame. There is a bijection between the round filters of $L$ and the $R$-increasing closed subsets of $X=\Prim(L)$, given by 
\begin{equation}\label{def_of_Phi}\Phi: \F  \longmapsto F_{\F} := \lbrace x \in X \mid x \supseteq \F \rbrace. \end{equation}
\end{proposition}
\begin{proof}
First, it is clear that $F_{\F}$ is a closed set and that it is $R$-increasing, so that $\Phi$ is well defined. Moreover, $\Phi$ is one-to-one since every filter is the intersection of the prime filters containing it. Finally, we show that $\Phi$ is onto. Let $F$ be an $R$-increasing closed set. In particular, $F$ is an increasing\footnote{for the order of $X$} closed subset (recall Definition \ref{def_ordered_Gleason_space}), and, therefore, we know that 
\[ F = \bigcap \lbrace \eta(a) \mid \eta(a) \supseteq F \rbrace. \]  If we set $\F= \lbrace a \mid \eta(a) \supseteq F \rbrace$, then

\begin{align*}
x \in F_{\F} \Leftrightarrow \R \subseteq x \Leftrightarrow (\eta(a) \supseteq F \Rightarrow x \in a) \Leftrightarrow x \in F.
\end{align*}Since one can show that $\F$ is a filter by routine calculations, it remains to prove it is round.

 Let $a$ be an element of $\F$. As $F$ is an $R$-increasing set, it comes that $R[F,-] \subseteq \eta(a)$. Recall that $R$ is a closed relation and that $\lbrace \eta(b) \mid \eta(b) \supseteq F \rbrace$ is a filtered family of closed sets such that $F = \bigcap \lbrace \eta(b) \mid \eta(b) \supseteq F \rbrace$. Hence, by Esakia Lemma (see for instance \cite[p. 995]{Sambin1}), it follows that
\[ \eta(a) \supseteq R[F,-] = R[\bigcap \eta(b),-] = \bigcap R[\eta(b),-]. \]
It is now sufficient to use compactness to  obtain 
\[ R[ \eta(b_1) \cap \cdots \cap \eta(b_n),-] \subseteq R[\eta(b_1),-] \cap \cdots \cap R([\eta(b_n),-] \subseteq \eta(a) \] for some $b_1,\ldots, b_n$. If we set $b:= b_1 \wedge \cdots \wedge b_n$, we have 
\[ F \subseteq \eta(b) \text{ and } R[\eta(b),-] \subseteq \eta(a), \] that is $b \in \F$ and $b \prec a$ as required.
\end{proof}

Let us note that the application $\Phi$ defined in \eqref{def_of_Phi} is a reverse order isomorphism, in the sense that for two round filters $\F$ and $\F'$, we have $\F \subseteq \F'$ if and only $\Phi({\F}) \supseteq \Phi({\F'})$. Therefore, the $R$-increasing closed sets which are associated to ends are exactly the join-prime $R$-increasing closed sets. We will use this observation and the next definition to prove the reciprocal of Lemma \ref{lem_precx=round}.

\begin{definition}\label{def_of_equiv} Let $(X,\leq,R)$ be an ordered Gleason space. We denote by $\equiv$ the equivalence relation associated to the pre-order $R$, i.e.
\[ x \equiv y \text{ if and only if } x \mathrel{R} y \text{ and } y \mathrel{R} x. \]
Since $R$ is closed, $\equiv$ is also closed. Moreover,  $X/\equiv$ ordered by
\[ x^\equiv \leq_R y^\equiv \text{ if and only if } x \mathrel{R} y \] is a compact pospace.

We highlight the fact that, if $(X,\leq,R)$ is the dual of a proximity frame $(L,\prec)$, then the equivalence relation $\equiv$ can be expressed as follow: 
\[ x \equiv y \text{ if and only if } \Uprec x = \Uprec y, \] or, equivalently, 
\[ x \equiv y \text{ if and only if } \Dprec x^c = \Dprec y^c. \]
\end{definition}

\begin{lemma}\label{lem_exist_x_mini_in_f_u} Let $(L,\prec)$ be a proximity frame. If $p \in \End(L)$, then there exists a unique $\equiv$-class $x^\equiv$ such that $x$ is R-minimal $x$ and $F_p = R[x,-]$. 
\end{lemma}
\begin{proof}
We know that for every element $z \in F_p$, there exists an $R$-minimal element $x \in F_p$ such that $x \mathrel{R} z$. Hence, it remains to prove its uniqueness.

Suppose that there exist two $R$-minimal elements $x$ and $y$ in $F_p$ such that $x \not\equiv y$. In other words, we have $x \not\mathrel{R} y$ and $x \not\mathrel{R} y$. Using a classical argument, one can show that there exist two $R$-decreasing open sets $\omega_1$ and $\omega_2$ such that $x \in \omega_1$, $y \in \omega_2$ and $\omega_1 \cap \omega_2 = \emptyset$. In other words, such that $F_p = F_p \setminus \omega_1 \cup F_p \setminus \omega_2$. Since $F_p \setminus \omega_i$ are $R$-decreasing closed sets and $F_p$ is join-prime (recall the discussion after Proposition
\ref{prop_round_filter_and_r_increasing_set}), it follows that $F_p \subseteq F_p \setminus \omega_1$ or $F_p \subseteq F_p \setminus \omega_2$, which is of course impossible since, for instance, $x \in F_p $ and $x \in \omega_1 $.  
\end{proof}

\begin{theorem}\label{thm_sigma_onto} Let $(L,\prec)$ be a proximity frame. A subset $p \subseteq L$ is an end if and only if $ p = \Uprec x$ for some $x \in \Prim(L)$.  
\end{theorem}
\begin{proof} The if part is Lemma \ref{lem_precx=round}. For the only part, let $p$ be an end. By Lemma \ref{lem_exist_x_mini_in_f_u}, we have $F_p = R[x,-]$ for some prime filter $x$. In particular, it follows that $\Phi(p) = \Phi(\Uprec x)$ and therefore that $p = \Uprec x$, as required. \end{proof}

It follows from Theorem \ref{thm_sigma_onto} that, at least for the underlying sets, $\End(L)$ is the quotient of $\Prim(L)$ by the relation $\equiv$. We denote by $\sigma$ the application 
\[ \sigma : \Prim(L)/{\equiv} \longrightarrow \End(L) : x^\equiv \longmapsto \Uprec x. \]  We now want to prove that $\End(L)$ is the quotient of $\Prim(L)$ as ordered topological spaces, that is prove that $\sigma$ is an order homeomorphism.

\begin{theorem}\label{thm_plop} Let $(L,\prec)$ be a proximity frame. Then, in the category \textbf{KPSp}, we have 
\[ \End(L) \cong \Prim(L)/{\equiv}, \] by the application $\sigma$.
\end{theorem}
\begin{proof}
First, we have that $\sigma$ is onto by Theorem \ref{thm_sigma_onto} and that it is one-to-one by definition. We also have that $\sigma$ is an order isomorphism, since we have 
\[ x^\equiv \leq_R y^\equiv \Leftrightarrow x \mathrel{R} y \Leftrightarrow \Uprec x \subseteq \Uprec y \Leftrightarrow \sigma(x^\equiv) \leq \sigma(y^\equiv). \] Therefore, since $\End(L)$ and $\Prim(L)/{\equiv}$ are compact Hausdorff spaces, it remains to prove $\sigma$ is continuous. By Proposition \ref{prop_carac_of_P_as_KPSP}, we have to prove that $\sigma^{-1}(\mu(a))$  and  $\sigma^{-1}(K_{\F})$ are respectively open and closed for every $a \in L$ and $\F \in \mathcal{RI}(L)$. Let $\Pi: x \longmapsto x^\equiv$ be the canonical quotient map. We have: 
\[ \Pi^{-1}(\sigma^{-1}(\mu(a))) = \lbrace x \in \Prim(L) \mid a \in \Uprec x \rbrace = \bigcup \lbrace \eta(b) \mid b \prec a \rbrace \] which is open, and 
\[ \Pi^{-1}(\sigma^{-1}(K_{\F})) = \lbrace x \in \Prim(L) \mid \F \subseteq x \rbrace = F_{\F}, \] which is closed.  
\end{proof}

With Theorem \ref{thm_plop}, we can describe a functor from the category \textbf{OGlSp} to the category \textbf{KPSp} which sends an ordered Gleason space $(X,\leq,R)$ to the compact pospace $(X/{\equiv},R)$. Proposition \ref{prop_ok_funct} gives a hint on how to deal with the morphisms. Indeed, if $\rho \subseteq X \times Y$ is a strong meet-hemirelation which satisfies the ofc and the dvc between ordered Gleason spaces, then we know that it can be associated with a proximity morphism $h : \uOf(Y) \longrightarrow \uOf(X): O \longmapsto
 \rho[-,O^c]^c$. Then, this morphism is associated with the continuous function
 \[ f : \End(\uOf(X)) \longrightarrow \End(\uOf(Y)) : p \longmapsto \Uprec h^{-1}(p). \]
Now, $p$ is equal to $\Uprec x$ for some $x \in X$ and if $y$ is an $R$-minimal element in $\rho[x,-]$, that is such that $h^{-1}(x) \subseteq y$, we have 
\[ f(p) = f(\Uprec x) = \Uprec h^{-1}(\Uprec x) = \Uprec y. \]
Hence, we have to send a meet hemirelation $\rho$ to the function \[f_\rho :X/{\equiv} \longrightarrow Y/{\equiv}: x^\equiv \longmapsto y^\equiv \text{ for $y$ $R$-minimal $\rho[x,-]$}.\]  By the dualities between \textbf{KPSp} and \textbf{PrFrm} and between \textbf{PrFrm} and \textbf{OGlSp}, $f$ is an increasing continuous function. But, we have a direct proof.

\begin{proposition}\label{prop_continuous_function} Let $\rho \subseteq X \times Y$ be a strong hemirelation that satisfies ofc and dvc between two ordered Gleason spaces. The map defined by 
\[ f : X/{\equiv} \longrightarrow Y/{\equiv} :  x^\equiv \longmapsto y^\equiv \] for $y$ R-minimal in $\rho[x,-]$ is an increasing continuous function.
\end{proposition}
\begin{proof}
First, the ofc implies that $f$ is well defined and increasing. Now, since $Y/{\equiv}$ is a compact pospace, to prove that $f$ is continuous, it is enough to prove that $f^{-1}(\omega)$ and $f^{-1}(F)$ are respectively open and closed subsets of $X/{\equiv}$ for $\omega$ an open downset and $F$ a closed downset of $Y/{\equiv}$. 

For $\omega$, we have
\begin{align}
&x^\equiv \in f^{-1}(\omega) \nonumber \\
\iff & \exists y \in \Pi^{-1}(\omega) : y \text{ is $R$-minimal in } \rho[x,-] \label{plop} \\
\iff & x \in \rho[-,\Pi^{-1}(\omega)] \label{plopdeux}
\end{align}
While the implication $\eqref{plop} \Rightarrow \eqref{plopdeux}$ does not need to be proved, a word must be spent on the reciprocal. Suppose that $x \in \rho[-,\pi^{-1}(\omega)]$, then we have $x \mathrel{\rho} y$ for some $z \in \Pi^{-1}(\omega)$. By Proposition \ref{prop_exist_R_min}, there exists $y$ R-minimal in $\rho[x,-]$ such that $y \mathrel{R} z$. Now, $\omega$ is a downset of $Y/{\equiv}$, so that $\Pi^{-1}(\omega)$ is an $R$-decreasing subset of $Y$. It follows that $y \in \Pi^{-1}(\omega)$, as required. Restarting from \eqref{plopdeux}, since $\Pi^{-1}(\omega)$ is $R$-decreasing and open, it is in particular an open downset of $Y$. Therefore, 
\[ \Pi^{-1}(\omega) = \bigcup \lbrace O \in \dOf(Y) \mid O \subseteq \Pi^{-1}(\omega) \rbrace\] and, consequently,
\[ \rho[-,\Pi^{-1}(\omega)] = \bigcup \lbrace \rho[-,O] \mid O \subseteq \Pi^{-1}(\omega) \rbrace. \] By the dvc (see Proposition \ref{Prop_towards_ofc}), $\rho[-,O]$ is an open subset of $X$ for every $O \in \dOf(Y)$ and, hence, so is $\rho[-,\Pi^{-1}(\omega)]$. Henceforth, we proved that $\Pi^{-1}(f^{-1}(\omega)) = \rho[-,\Pi^{-1}(\omega)]$ is open in $X$, as required. 

Finally, as for $\omega$, we have that 
\[ \Pi^{-1}(f^{-1}(F)) = \rho[-,\Pi^{-1}(F).] \]
Now, since $\rho$ is closed in $X \times Y$, $\rho[-,\Pi^{-1}(F)]$ is a closed subset of $X$ and the proof is concluded. 
\end{proof}

Hence, we have a functor $\xi$ between the categories \textbf{OGlSp} and \textbf{KPSp} which maps an ordered Gleason space $(X,\leq,R)$ to the compact pospace $(X/{\equiv}, \leq_R)$ an an ordered Gleason relation $\rho \subseteq X \times Y $ to the increasing continuous function $f : X/{\equiv} \longrightarrow Y/{\equiv}$ defined in Proposition \ref{prop_continuous_function}. This functor yields an equivalence between \textbf{OGlSp} and \textbf{KPSp} which is equivalent to the composition of the duality between \textbf{OGlSp} and \textbf{PrFrm} and the duality between \textbf{PrFrm} and \textbf{KPSp}.
\begin{remark} In the Boolean setting, an important feature of Gleason spaces is that their underlying stone spaces are the projective objects in the category \textbf{KHaus}.  This is not the case anymore in our distributive setting. Indeed, the $f$-spaces are not the projective objects of the category \textbf{KPSp} since it would implies that they are projective in the category \textbf{Priest}. However, the injective objects of \textbf{DLat} have been shown in \cite{Balbes} to exactly be the complete Boolean algebras and not the frames. In fact, the projective objects of \textbf{KPSp} are exactly the projective objects of \textbf{KHaus}, that is the extremally disconnected compact spaces, as we show in the next proposition. 
\end{remark}

\begin{proposition} The projective objects in the category \textbf{KPSp} are exactly the extremally disconnected spaces (ordered by equality).
\end{proposition}
\begin{proof}
First, let us consider $(X,=)$ a extremally disconnected space, $(P,\leq)$ and $(Q,\leq)$ compact po-spaces, $f: X \longrightarrow P$ a monotone continuous function and $g: Q \longrightarrow P$ a surjective monotone continuous function. Since every compact po-space is in particular compact Hausdorff, and since the extremally disconnected spaces are projective in \textbf{KHaus}, there exists a continuous function $h : X \longrightarrow Q$ such that $gh = f$ and, since $X$ is ordered by the equality, $h$ is clearly monotone. Hence, $(X,=)$ is indeed projective in \textbf{KPSp}.

On the other hand, suppose that $(X,\tau,\leq)$  is projective in \textbf{KPSp}. Then, following the proof of Gleason in \cite{Gleason}, one can prove that $(X,\tau)$ is extremally disconnected.
\end{proof}

The main ``ethical'' reason behind the failure of ordered Gleason spaces as projective objects in \textbf{KPSp} is the the relation $R$ is submerged by its associated equivalence relation $\equiv$. A solution could be to change the properties of the morphisms in the projective problem so that they directly take into account $R$ instead of $\equiv$.

%\section{Work in progress}
%
%Non.
%
%\begin{definition} An increasing continuous function $f : X \longrightarrow Y$ is said to be \emph{irreducible} if for every proper closed set of $X$, $f(F)$ is a proper subset of $Y$.
%\end{definition}
%
%\begin{proposition} Let $(L,\prec)$ be a proximity frame and $(X,\leq,R)$ its space of prime filters. The quotient map 
%\[ \Pi : X \longrightarrow X/{\equiv}: x \longmapsto x^\equiv \] is irreducible.
%\end{proposition}
%\begin{proof}
%Let $F$ be a proper subset of $X$ , without loss of generality, we may consider it is not empty. We know that there exist $a,b \in L$ such that $F \subseteq \eta(a) \cap \eta(b)^c$ and, since $F$ is proper, we have $a \neq 1$ or $b \neq 0$. 
%
%Suppose first that $a \neq 1$. Then, by S8 and S1, we have $a \prec c \prec 1$ for some 
%\end{proof}
%
%
%\begin{definition} \emph{ Extended KPSp} are triple $(P,\leq, R)$ where $(P,\leq)$ is a KPSp and $R$ is a \emph{good relation}. 
%
%An \emph{extended function} between extended KPSp $f:(X,\leq,R) \longrightarrow (Y,\leq,S)$ such that 
%\[ x_1 \mathrel{R} x_2 \iff f(x_1) \mathrel{S}  f(x_2). \]
%\end{definition}
%
%\begin{proposition} Let $(X,\leq,R)$ be an extended f-space, $(Q,\leq,S)$ an extended KPSp
%\end{proposition}

\section*{Conclusion}

We have completed the external network of equivalences and dualities started in \cite{Guramtriang} and \cite{DeHaGelfand}, generalising to the "distributive setting" the duality between Gleason spaces and compact Hausdorff spaces of \cite{Sourabh}. Hence, we obtain the following commutative diagram, where the arrowed lines represent adjunctions and the non-arrowed ones equivalences or dualities.

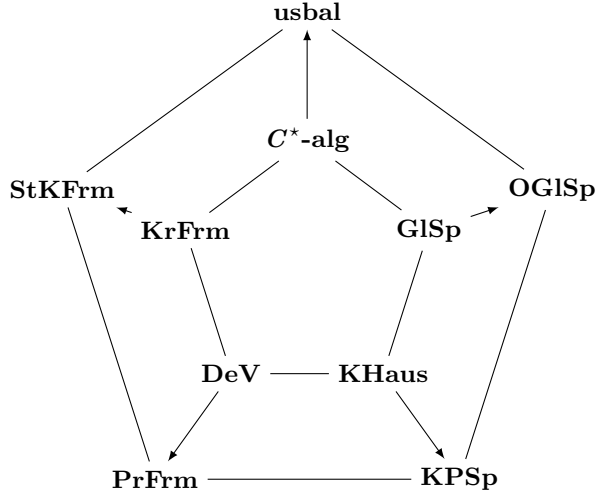
\begin{center}
\begin{tikzpicture}
\node(KHaus) at (-1,0) {\textbf{DeV}};
\node(KPSp) at (-2,-1.4) {\textbf{PrFrm}};
\node(DeV) at (1,0) {\textbf{KHaus}};
\node(PrFrm) at (2,-1.4) {\textbf{KPSp}};
\node(C) at (0,3.1) { \textbf{\emph{C}$^\star$-alg}};
\node(GlSp) at (1.6,1.9) { \textbf{GlSp}};
\node(KrFrm) at (-1.6,1.9) {\textbf{KrFrm}};
\node(StKFrm) at (-3.22,2.44) {\textbf{StKFrm}};
\node(OGlSp) at (3.22,2.44) {\textbf{OGlSp}};
\node(usbal) at (0,4.8) {\textbf{usbal}};
\draw[-, >=latex] (KHaus) to (DeV) ;
\draw[-, >=latex] (DeV) to(GlSp);
\draw[-, >=latex] (KrFrm) to(KHaus);
\draw[-, >=latex] (KrFrm) to (C);
\draw[-, >=latex] (GlSp) to (C);
\draw[-, >=latex] (KPSp) to (PrFrm) ;
\draw[-, >=latex] (PrFrm) to(OGlSp);
\draw[-, >=latex] (StKFrm) to(KPSp);
\draw[-, >=latex] (StKFrm) to (usbal);
\draw[-, >=latex] (OGlSp) to (usbal);
\draw[->, >=latex] (C) to (usbal);
\draw[->, >=latex] (GlSp) to (OGlSp);
\draw[->, >=latex] (KrFrm) to (StKFrm);
\draw[->, >=latex] (KHaus) to (KPSp);
\draw[->, >=latex] (DeV) to (PrFrm);
\end{tikzpicture}
\end{center}
However, A proper way to describe the functor between \textbf{KPSp} and \textbf{OGlSp} is still missing. This situation could be solved figuring out the universal problem answered by ordered Gleason spaces. This problem cannot be the usual projective one as we saw at the end of Section \ref{Section3}. We will address this problem in a forthcoming article.


\begin{thebibliography}{10}

\bibitem{Balbes}
R.~{Balbes}.
\newblock {Projective and injective distributive lattices}.
\newblock {\em {Pac. J. Math.}}, 21:405--420, 1967.

\bibitem{StonebyDV}
Guram {Bezhanishvili}.
\newblock {Stone duality and Gleason covers through de Vries duality}.
\newblock {\em {Topology and its Applications}}, 157(6):1064--1080, 2010.

\bibitem{Sourabh}
Guram {Bezhanishvili}, Nick {Bezhanishvili}, Sumit {Sourabh}, and Yde {Venema}.
\newblock {Irreducible equivalence relations, Gleason spaces and de Vries
  duality}.
\newblock {\em {Applied Categorical Structures}}, 25(3):381--401, 2017.

\bibitem{Guramtriang}
Guram {Bezhanishvili} and John {Harding}.
\newblock {Proximity Frames and Regularization}.
\newblock {\em {Applied Categorical Structures}}, 22:43--78, 2014.

\bibitem{Gurambal}
Guram {Bezhanishvili}, Patrick~J. {Morandi}, and Bruce {Olberding}.
\newblock {Bounded archimedean $\ell$-algebras and Gelfand-Neumark-Stone
  duality}.
\newblock {\em Theory and Applications of Categories}, 28(16):435--475, 2013.

\bibitem{Castro}
Jorge {Castro} and Sergio {Celani}.
\newblock {Quasi-modal lattices}.
\newblock {\em {Order}}, 21(2):107--129, 2004.

\bibitem{CelaniResume}
Sergio {Celani}.
\newblock {Precontact relations and quasi-modal operators in Boolean algebras}.
\newblock In {\em {Actas del XIII congreso ``Dr. Antonio A. R. Monteiro''}},
  pages 63--79. Bah\'{\i}a Blanca: Universidad Nacional del Sur, Instituto de
  Matem\'atica, 2016.

\bibitem{DeHaGelfand}
Laurent {De Rudder} and Georges {Hansoul}.
\newblock {A Gelfand duality for compact pospaces}.
\newblock {\em {Algebra Universalis}}, 79(2):13, 2018.
\newblock Id/No 47.

\bibitem{DeHa}
Laurent {De Rudder}, Georges Hansoul, and Valentine Stetenfeld.
\newblock Subordination algebras in modal logic, 2020.

\bibitem{deVries}
Hendrik {de Vries}.
\newblock {\em {Compact spaces and compactications. An algebraic approach}}.
\newblock PhD thesis, Universiteit van Amsterdam, 1962.

\bibitem{Compendium}
Gerhard {Gierz}, Karl {Hofmann}, Klaus {Keimel}, Jimmie {Lawson}, Michael
  {Mislove}, and Dana~S. {Scott}.
\newblock {\em {Continuous lattices and domains}}, volume~93 of {\em
  Encyclopedia of Mathematics and its Applications}.
\newblock Cambridge university press, 2003.

\bibitem{Gleason}
Andrew~M. {Gleason}.
\newblock {Projective topological spaces}.
\newblock {\em{Illinois Journal of Mathematics}}, 2:482--489, 1958.

\bibitem{Koppelberg}
Sabine {Koppelberg}, Ivo {D\"{u}ntsch}, and Michael {Winter}.
\newblock {Remarks on contact relations on Boolean algebras}.
\newblock {\em Algebra Universalis}, 68:353--366, 2012.

\bibitem{Priestley1}
Hillary~A. {Priestley}.
\newblock {Representation of distributive lattices by means of ordered Stone
  spaces}.
\newblock {\em Bulletin of the London Mathematical Society}, 2:186--190, 1970.

\bibitem{Pultr}
Ale\v{s} {Pultr} and J.~{Sichler}.
\newblock {Frames in Priestley's duality}.
\newblock {\em {Cahiers de Topologie et G\'eom\'etrie Diff\'erentielle. Cat\'egoriques}},
  29(3):193--202, 1988.

\bibitem{Sambin1}
Giovanni {Sambin} and Virginia {Vaccaro}.
\newblock {A new proof of Sahlqvist's theorem on modal definability and
  completeness}.
\newblock {\em {Journal of Symbolic Logic}}, 54(3):992--999, 1989.

\bibitem{Sofronie}
Viorica {Sofronie-Stokkermans}.
\newblock {Duality and canonical extensions of bounded distributive lattices
  with operators, and applications to the semantics of non-classical logics.
  I.}
\newblock {\em {Studia Logica}}, 64(1):93--132, 2000.

\end{thebibliography}
\end{document}